\documentclass{article}
\usepackage{graphicx} 
\usepackage{amssymb}
\usepackage{mathrsfs}
\usepackage{amsmath}
\usepackage{amsthm}
\newtheoremstyle{mytheoremstyle}
  {3pt}
  {3pt}
  {\normalfont}
  {0pt}
  {\bfseries}
  {}
  {.5em}
  {}
\theoremstyle{mytheoremstyle}
\newtheorem{theorem}{Theorem}[section]

\newtheorem{lemma}{Lemma}[section]
\newtheorem{proposition}{Proposition}[section]

\newtheorem{claim}{Claim}

\usepackage[urlcolor=blue]{hyperref}
\usepackage{threeparttable}
\usepackage{setspace}
\usepackage{titlesec}
\usepackage{float}
\usepackage{fancyhdr}
\usepackage{cite}
\usepackage{tikz}
\usetikzlibrary{arrows.meta,positioning}
\usetikzlibrary{fit}
\usepackage{booktabs} 
\usepackage{pgffor} 

\title{Internally-disjoint Pendant Steiner Trees in Digraphs}

\author{Shanshan Yu$^{1}$,  Yuefang Sun$^{2}$
\\
$^{1}$ School of Mathematics and Statistics,
Ningbo University,\\
Ningbo 315211,  China, yushanshan33hh@163.com\\
$^{2}$ Corresponding author. School of Mathematics and Statistics,\\
Ningbo University,
Ningbo 315211, China, sunyuefang@nbu.edu.cn}
\date{}

\begin{document}

\maketitle

\begin{abstract}
      
    For a digraph $D=(V(D),A(D))$ and a set $S\subseteq V(D)$ with $|S|\geq 2$ and $r\in S$, a directed pendant $(S,r)$-Steiner tree (or, simply, a pendant $(S,r)$-tree) is an out-tree $T$ rooted at $r$ such that $S\subseteq V(T)$ and each vertex of $S$ has degree one in $T$. Two pendant $(S,r)$-trees are called internally-disjoint if they are arc-disjoint and their common vertex set is exactly $S$. The goal of the {\sc Internally-disjoint Directed Pendant Steiner Tree Packing (IDPSTP)} problem is to find a largest collection of pairwise internally-disjoint pendant $(S,r)$-trees in $D$. Let $\tau_{k}(D)=\min\{\tau_{S,r}(D)\mid S\subseteq V(D),|S|=k,r\in S\}$, where $\tau_{S,r}(D)$ denotes the maximum number of pairwise internally-disjoint pendant $(S,r)$-trees in $D$.

   
    In this paper, we first completely determine the computational complexity for the decision version of IDPSTP on Eulerian digraphs and symmetric digraphs. We then show that, for any $\epsilon>0$, given an instance of IDPSTP with order $n$, it is NP-hard to approximate the solution within $O(n^{{1/3}-\epsilon})$. 
    Finally, we get some sharp bounds for the parameter $\tau_{k}(D)$.
\vspace{0.3cm} 
\vspace{0.3cm}
\end{abstract}

\section{Introduction}

\subsection{Notation and terminology}
All digraphs considered in this paper have no parallel arcs or loops. We use $[n]$ to denote the set of all natural numbers from $1$ to $n$, and let $|S|$ denote the number of elements in a set of $S$. 

Let $D=(V(D),A(D))$ be a digraph. For a vertex $v\in V(D)$, we use $N_{D}^{+}(v)$ (resp. $N_{D}^{-}(v)$) to denote the {\em out-neighbours} (resp. {\em in-neighbours}) of $v$ in $D$, that is, $$N_{D}^{+}(v)=\{u\in V(D)\backslash \{v\}\mid vu\in A(D)\}, N_{D}^{-}(v)=\{w\in V(D)\backslash \{v\}\mid wv\in A(D)\}.$$ 
Furthermore, for a subset $W\subseteq V(D)$, we use $N_{D}^{+}(W)$ (resp. $N_{D}^{-}(W)$) to denote the out-neighbours (resp. in-neighbours) of $W$ in $D$, that is, $$N_{D}^{+}(W)=\bigcup_{w\in W}N_{D}^{+}(w)\backslash W,\quad N_{D}^{-}(W)=\bigcup_{w\in W}N_{D}^{-}(w)\backslash W.$$
The {\em out-degree} (resp. {\em in-degree}) of $v$ in $D$ is $d_{D}^{+}(v)=|N_{D}^{+}(v)|$ (resp. $d_{D}^{-}(v)=|N_{D}^{-}(v)|$). We call out-degree and in-degree of a vertex its {\em semi-degrees}. The {\em degree} of a vertex $v$ in $D$ is the sum of its semi-degrees, i.e., $d_{D}(v)=d_{D}^{+}(v)+d_{D}^{-}(v)$.
If the context is clear, we always omit $D$ in the above notation. 
The {\em minimum out-degree} (resp. {\em minimum in-degree}) of $D$ is $\delta^{+}(D)=\min\{d^{+}(v)\mid v\in V(D)\}$ (resp. $\delta^{-}(D)=\min\{d^{-}(v)\mid v\in V(D)\}$). The {\em minimum semi-degree} of $D$ is $\delta^{0}(D)=\min\{{\delta}^{+}(D),{\delta}^{-}(D)\}$.

A digraph $D$ is {\em Eulerian} if $D$ is connected and $d^+(x)=d^-(x)$ for every vertex $x\in V(D)$. A digraph $D$ is {\em symmetric}, if $xy \in A(D)$ then $yx \in A(D)$. In other words, a symmetric digraph $D$ can be obtained from its underlying undirected graph $G$ by replacing each edge of $G$ with the corresponding arcs of both directions, denoted by $D=\overleftrightarrow{G}$. A digraph is {\em connected} if its underlying undirected graph is connected. Observe that a connected symmetric digraph is Eulerian.  

\subsection{Backgrounds}
The packing problem is not only a fundamental problem in combinatorics but also one of the important research topics in theoretical computer science. The goal of a graph packing problem is to find edge-disjoint subgraphs which satisfy specific structural conditions in a given graph. This type of graph-theoretical problem includes the famous graph matching problem~\cite{Galil-Micali-Gabow}, Hamiltonian cycle decomposition problem~\cite{Tillson}, disjoint paths problem~\cite{Kawarabayashi-Kobayashi-Reed} and Steiner tree packing problem~\cite{Sun-Yeo}. 

For a graph $G=(V(G), E(G))$ and a set $S\subseteq V(G)$ with $|S|=k\geq 2$, a tree $T$ of $G$ is called an {\em $S$-Steiner tree} (or {\em $S$-tree} for short) if $S\subseteq V(T)$. Two $S$-trees $T_{1}$ and $T_{2}$ are said to be {\em edge-disjoint} if $E(T_{1})\cap E(T_{2})=\varnothing$; moreover, they are {\em internally-disjoint} if $V(T_{1})\cap V(T_{2})=S$. The famous (undirected) \textbf{Steiner tree packing problem} is defined as follows: Given an undirected graph $G$ and a terminal set $S\subseteq V(G)$, the goal is to find the largest collection of pairwise edge-disjoint $S$-Steiner trees. Over the years, this problem has received considerable attention from the researchers of graph theory, combinatorial optimization and theoretical computer science~\cite{Gr¨otschel-Martin-Weismantel-1,Gr¨otschel-Martin-Weismantel-2,Gr¨otschel-Martin-Weismantel-3,Jain-Mahdian-Salavatipour,Li-Mao,Pulleyblank}, due to its important theoretical significance and strong background of applications.. 

The concept of $S$-Steiner tree was extended to digraphs by researchers. An {\em out-tree} is an oriented tree in which every vertex except one, called the {\em root}, has in-degree one. For a digraph $D=(V(D),A(D))$, and a subset of vertex $S$ with $r\in S$ and $|S|=k\geq 2$, a {\em directed $(S,r)$-Steiner tree} or, simply, {\em an $(S,r)$-tree}, is an out-tree $T$ rooted at $r$ with $S\subseteq V(T)$. Two $(S,r)$-trees $T_{1}$ and $T_{2}$ are said to be {\em internally-disjoint} if $A(T_{1})\cap A(T_{2})=\varnothing$ and $V(T_{1})\cap V(T_{2})=S$.

Cheriyan and Salavatipour \cite{Che-Sal}, and Sun and Yeo \cite{Sun-Yeo} studied the following \textbf{Internally-disjoint Directed Steiner Tree packing (IDSTP)} problem: Given a digraph $D=(V(D),A(D))$ and a terminal set $S\subseteq V(D)$ with a root $r\in S$, the goal is to find a largest collection of pairwise internally-disjoint $(S,r)$-trees. 

Furthermore, Sun and Yeo \cite{Sun-Yeo} defined the {\em generalized $k$-vertex-strong connectivity} (which is also called the {\em directed tree connectivity}) of $D$ as
\begin{align*}
    \kappa_{k}(D)=\min\{\kappa_{S,r}(D)\mid S\subseteq V(D),|S|=k,r\in S\},
\end{align*}
where $\kappa_{S,r}(D)$ denotes the maximum number of pairwise internally-disjoint $(S,r)$-trees in $D$. Clearly, $\kappa_{2}(D)=\kappa(D)$, where $\kappa(D)$ is the classical vertex-strong connectivity of $D$. Some results on directed Steiner tree packing and related topics have been obtained in \cite{Sun-Gutin-Yeo-Zhang, Sun-Yeo} (also can be found in a new monograph~\cite{Sun-book}). 

In this paper, we introduce the internally-disjoint directed pendant Steiner tree packing problem, which could be seen as a restriction of IDSTP. If each vertex of $S$ has degree one in an $(S,r)$-tree $T$, then $T$ is called a {\em pendant $(S,r)$-tree}. 
Two pendant $(S,r)$-trees $T_{1}$ and $T_{2}$ are said to be {\em internally-disjoint} if $A(T_{1})\cap A(T_{2})=\varnothing$ and $V(T_{1})\cap V(T_{2})=S$. Observe that in a pendant $(S, r)$-tree,  the in-degree of $r$ and the out-degree of each vertex $v\in S\setminus \{r\}$ is zero, and the out degree of $r$ and the in-degree of each vertex $v\in S\setminus \{r\}$ is one.  For example, in Figure~\ref{fig:two steiner trees}, there are two $(S, r)$-trees with $S=\{r, v_1, v_2\}$, one is a pendant $(S, r)$-tree and the other one is not pendant.

\begin{figure}[h]
    \centering
    \begin{tikzpicture}
        \begin{scope}[xshift=0cm]
            \node[circle,draw,fill=white,inner sep=1.5pt,line width=0.8pt] (r) at (0,0) {};
            \node[circle,draw,fill=white,inner sep=1.5pt,line width=0.8pt] (u) at (1,0) {};
            \node[circle,draw,fill=white,inner sep=1.5pt,line width=0.8pt] (v1) at (2,0.8) {};
            \node[circle,draw,fill=white,inner sep=1.5pt,line width=0.8pt] (v2) at (2,-0.8) {};
            \node at (1,-1.5) {(a) A pendant $(S, r)$-tree};

            \draw[->] (r) -- (u);
            \draw[->] (u) -- (v1);
            \draw[->] (u) -- (v2);

            \node[below]  at (r) {$r$};
            \node[below]  at (u) {$u$};
            \node[right]  at (v1) {$v_{1}$};
            \node[right]  at (v2) {$v_{2}$};
        \end{scope}

        \begin{scope}[xshift=5cm]
            \node[circle,draw,fill=white,inner sep=1.5pt,line width=0.8pt] (r) at (0,0) {};
            \node[circle,draw,fill=white,inner sep=1.5pt,line width=0.8pt] (v1) at (2,0.8) {};
            \node[circle,draw,fill=white,inner sep=1.5pt,line width=0.8pt] (v2) at (2,-0.8) {};
            \node at (1,-1.5) {(b) A non-pendant $(S, r)$-tree};

            \draw[->] (r) -- (v1);
            \draw[->] (r) -- (v2);

            \node[below]  at (r) {$r$};
            \node[right]  at (v1) {$v_{1}$};
            \node[right]  at (v2) {$v_{2}$};
        \end{scope}
    \end{tikzpicture}
    \caption{Two $(S, r)$-trees with $S=\{r, v_1, v_2\}$}
    \label{fig:two steiner trees}
\end{figure}


The \textbf{Internally-disjoint Directed Pendant Steiner Tree Packing (IDPSTP)} problem can be defined as follows: Given a digraph $D=(V(D),A(D))$ and a terminal set $S\subseteq V(D)$ with a root $r\in S$, the goal is to find a largest collection of pairwise internally-disjoint pendant $(S,r)$-trees. 

Now we define the directed pendant-tree $k$-connectivity of a digraph $D$ as
\begin{align*}
    \tau_{k}(D)=\min\{\tau_{S,r}(D)\mid S\subseteq V(D),|S|=k,r\in S\},
\end{align*}
where $\tau_{S,r}(D)$ denotes the maximum number of pairwise internally-disjoint pendant $(S,r)$-trees in $D$. By definition, it is clear that
\begin{align*}
    \begin{cases}
        \tau_{2}(D) = \kappa_{2}(D)=\kappa(D), \\
        \tau_{k}(D) \leq \kappa_{k}(D), 3\leq k\leq n.
    \end{cases}
\end{align*}
Hence, $\tau_{k}(D)$ is another type of generalization of the classical vertex-strong connectivity $\kappa(D)$ of digraphs. Note that this parameter is also an extension of pendant-tree connectivity~\cite{Hager} from undirected graphs to digraphs.

\subsection{Our results}

This paper is organized as follows. In Section 2, we study the computational complexity of the decision version of the internally-disjoint pendant Steiner tree packing problem in digraphs. When both $k\geq 3$, $\ell\geq 2$ are fixed integers, we prove that the problem of deciding whether there are at least $\ell$ pairwise internally-disjoint pendant $(S,r)$-trees in Eulerian digraphs is NP-complete, where $r\in S\subseteq V(D)$ and $|S|=k$ (Theorem \ref{thm2.2}). And this result implies the following Table~\ref{tab:Eulerian_digraphs}.

\begin{table}[htbp] 
    \centering 
    \caption{Eulerian digraphs}
    \begin{tabular}{l|c|c} 
    \toprule 
    $\tau_{S,r}(D) \geq \ell~?$, $|S| = k$ & $k \geq 3$ constant & $k$ part of input \\
    \midrule 
    $\ell \geq 2$ constant & NP-complete & NP-complete\\
    \midrule 
    $\ell$ part of input & NP-complete & NP-complete\\
    \bottomrule 
    \end{tabular}
    \label{tab:Eulerian_digraphs}
\end{table}

However, the complexity for the internally-disjoint pendant Steiner tree packing problem in symmetric digraphs is different. When both $k\geq 3$, $\ell\geq 2$ are fixed integers, the problem of deciding whether there are at least $\ell$ pairwise internally-disjoint pendant $(S,r)$-trees in symmetric digraphs is polynomial-time solvable (Theorem \ref{thm2.3}), and it becomes NP-complete otherwise (Theorems~\ref{thm2.4} and~\ref{thm2.6}), as shown in Table~\ref{tab:symmetric_digraphs}. 

\begin{table}[htbp]
    \centering 
    \caption{Symmetric digraphs}
    \begin{tabular}{l|c|c} 
    \toprule 
    $\tau_{S,r}(D) \geq \ell~?$, $|S| = k$  & $k \geq 3$ constant & $k$ part of input \\
    \midrule 
    $\ell \geq 2$ constant & Polynomial & NP-complete\\
    \midrule 
    $\ell$ part of input & NP-complete & NP-complete\\
    \bottomrule 
    \end{tabular}
    \label{tab:symmetric_digraphs}
\end{table}

In the last part of Section 2, we also study the inapproximability of IDPSTP (Theorem~\ref{inappro}) by showing that, for any $\epsilon>0$, given an instance of IDPSTP with order $n$, it is NP-hard to approximate the solution within $O(n^{{1/3}-\epsilon})$.

In Section 3, we study sharp bounds for the parameter $\tau_{k}(D)$. 
The first result is about a sharp lower bound and a sharp upper bound for $\tau_k(D)$ in terms of $n$ and $k$ (Theorem~\ref{thm3.1}). The second result concerns a sharp upper bound for $\tau_{k}(D)$ in terms of arc-cuts of $D$ (Theorem~\ref{thm3.2}). Finally, we obtain sharp Nordhaus-Gaddum type bounds of $\tau_{k}(D)$ (Theorem~\ref{thm3.3}).


\section{Computational complexity for $\tau_{S,r}(D)$}
In this section, we will first completely determine the computational complexity for the decision version of IDPSTP on Eulerian digraphs and symmetric digraphs, and then study the inapproximability of IDPSTP.

\subsection{Eulerian digraphs}
The problem of \textbf{Directed $k$-Linkage} is defined as follows: 
Let $k\geq 2$ be a fixed integer. Given a digraph $D$ and a (terminal) sequence $((s_1,t_1),\cdots ,(s_k,t_k))$ of distinct vertices of $D,$ decide whether $D$ has $k$ vertex-disjoint paths $P_1,\cdots ,P_k$, where $P_i$ starts at $s_i$ and ends at $t_i$ for each $i\in [k].$

Sun and Yeo proved the NP-completeness of Directed 2-Linkage for Eulerian digraphs.

\begin{theorem}\label{thm2.1}\cite{Sun-Yeo}
    The problem of Directed 2-Linkage restricted to Eulerian digraphs is NP-complete.
\end{theorem}

By Theorem~\ref{thm2.1}, we will now show that the problem of determining whether $\tau_{S,r}(D)\geq \ell$ on Eulerian digraphs (and therefore general digraphs) is NP-complete. And this result implies the contents of Table~\ref{tab:Eulerian_digraphs}.

\begin{theorem}\label{thm2.2}
    Let $D$ be an Eulerian digraph and $S\subseteq V(D)$ with $|S|=k$ and $r\in S$. Let $k\geq 3$ and $\ell\geq 2$ be fixed integers. The problem of deciding if $\tau_{S,r}(D)\geq \ell$ is NP-complete.
\end{theorem}

\begin{proof}
    Let $D^{*}$ be an Eulerian digraph and $s_{1},s_{2},t_{1},t_{2}$ be four distinct vertices in $D^{*}$. We construct a new digraph $D$ (from $D^{*}$) with
    \begin{align*}
        V(D)=&V(D^{*})\cup U\cup V\cup \{r,u_{1},u_{2}\},
    \end{align*}
    where $U=\{u_{3},u_{4},\cdots,u_{k-1}\}$ and $V=\{v_{i}\mid i\in [\ell-2]\}$. The arc-set of $D$ is defined as
    \begin{align*}
        A(D)=&A(D^{*})\cup\{rs_{1},rs_{2},t_{1}u_{1},u_{1}t_{1},t_{2}u_{2},u_{2}t_{2},s_{1}u_{2},s_{2}u_{1},u_{1}r,u_{2}r\}\\
             &\cup\{rv,vr,vu,uv\mid v\in V\mbox{ and }u\in U\}\\
             &\cup\{t_{i}u,ut_{i}\mid i\in[2]\mbox{ and }u\in U\}.
    \end{align*}

    \begin{figure}
        \centering
        \begin{tikzpicture}
            \draw (-0.5,-2) rectangle (4.5,1);
            \node[circle,draw,fill=white,inner sep=5pt,line width=0.8pt] (s1) at (0,0) {};
            \node[circle,draw,fill=white,inner sep=5pt,line width=0.8pt] (s2) at (0,-1) {};
            \node[circle,draw,fill=white,inner sep=5pt,line width=0.8pt] (t1) at (4,0.5) {};
            \node[circle,draw,fill=white,inner sep=5pt,line width=0.8pt] (t2) at (4,-1.5) {};
            \node at (2,-1.5) {$D^{*}$};
        
            \draw (6,-1.5) rectangle (7,0.5);
            \node[circle,draw,fill=white,inner sep=5pt,line width=0.8pt] (r) at (-2,-0.5) {};
            \node[circle,draw,fill=white,inner sep=5pt,line width=0.8pt] (u1) at (6.5,0) {};
            \node[circle,draw,fill=white,inner sep=5pt,line width=0.8pt] (u2) at (6.5,-1) {};

            \draw (-0.5,-3.5) rectangle (2.5,-4.5);
            \node[circle,draw,fill=white,inner sep=5pt,line width=0.8pt] (v1) at (0,-4) {};
            \node[circle,draw,fill=white,inner sep=5pt,line width=0.8pt] (v2) at (0.7,-4) {};
            \node[circle,draw,fill=white,inner sep=7pt,line width=0.8pt] (v3) at (2,-4) {};
            \node at (1.35,-4) {$\cdots$};
            \node at (1,-5) {$V$};
            
            \draw (4,-3.5) rectangle (7,-4.5);
            \node[circle,draw,fill=white,inner sep=5pt,line width=0.8pt] (u3) at (4.5,-4) {};
            \node[circle,draw,fill=white,inner sep=5pt,line width=0.8pt] (u4) at (5.2,-4) {};
            \node[circle,draw,fill=white,inner sep=7pt,line width=0.8pt] (u5) at (6.5,-4) {};
            \node at (5.85,-4) {$\cdots$};
            \node at (6,-5) {$U$};

            \draw (-2.5,-5.5) rectangle (7.5,2.5);
            \node at (7,1.5) {$D$};

            \draw[->] (r) -- (s1);
            \draw[->] (r) -- (s2);
            \draw[->] (s1) -- (u2);
            \draw[->] (s2) -- (u1);
            \draw[<->] (t1) -- (u1);
            \draw[<->] (t2) -- (u2);
            \draw[->] (u2) to [bend left=60] (r); 
            \draw[->] (u1) to [bend right=60] (r);
            \draw[<->] (t1) -- (6.5,-3.5);
            \draw[<->] (t2) -- (4.5,-3.5);
            \draw[<->] (r) to [bend right=30] (-0.5,-4);
            \draw[<->] (2.5,-4) -- (4,-4);
            \draw[<->] (1,-3.5) -- (6.7,-1.5);

            \node at (r) {$r$};
            \node at (s1) {$s_{1}$};
            \node at (s2) {$s_{2}$};
            \node at (t1) {$t_{1}$};
            \node at (t2) {$t_{2}$};
            \node at (u1) {$u_{1}$};
            \node at (u2) {$u_{2}$};
            \node at (v1) {$v_{1}$};
            \node at (v2) {$v_{2}$};
            \node at (v3) {$v_{\ell-2}$};
            \node at (u3) {$u_{3}$};
            \node at (u4) {$u_{4}$};
            \node at (u5) {$u_{k - 1}$};
        \end{tikzpicture}
        \caption{The Eulerian digraph $D$}
        \label{fig:thm2.2}
    \end{figure}
    
    This completes the construction of $D$ as shown in Figure \ref{fig:thm2.2}. Clearly, $D$ is Eulerian. Let $S=\{r, u_1, u_2\}\cup U$ and let $r$ be the root. We will show that $\tau_{S,r}(D)\geq \ell$ if and only if $[D^{*};s_{1},s_{2},t_{1},t_{2}]$ is a positive instance of the problem of Directed 2-Linkage on Eulerian digraphs, and then the result holds by Theorem~\ref{thm2.1}.

    First assume that $[D^{*};s_{1},s_{2},t_{1},t_{2}]$ is a positive instance of Directed 2-Linkage on Eulerian digraphs, that is, we have two vertex-disjoint paths, $P_{1}$ and $P_{2}$, such that $P_{i}$ is an $(s_{i},t_{i})$-path, for $i\in [2]$. 
    For each $i\in [\ell -2]$, let $T_{i}$ be the union of the arc $rv_{i}$ and all the arcs from $v_{i}$ to $\{u_{1},u_{2}\}\cup U$. Let $T_{\ell-1}$ be the union of the arcs $rs_{1}$ and $s_{1}u_{2}$, all the arcs from $t_{1}$ to $\{u_{1}\}\cup U$, and the path $P_{1}$. Let $T_{\ell}$ be the union of the arcs $rs_{2}$ and $s_{2}u_{1}$, all the arcs from $t_{2}$ to $\{u_{2}\}\cup U$, and the path $P_{2}$. It can be checked that these trees are $\ell$ pairwise internally-disjoint pendant $(S,r)$-trees, which means that $\tau_{S,r}(D)\geq \ell$.

    Conversely, assume that $\tau_{S,r}(D)\geq \ell$ and let $T_{1}',T_2',\cdots,T_{\ell}'$ be $\ell$ pairwise internally-disjoint pendant $(S,r)$-trees in $D$. By the definition of a pendant $(S,r)$-tree and the fact that the out-degree of $r$ is $\ell$, each $T_{i}'$ contains precisely one vertex from $\{s_1,s_2\}\cup V$. Hence, if we delete all vertices from $V$, then there are two pendant $(S,r)$-trees left, say $T_{1}'$ and $T_{2}'$. At this time, $u_{1}$ has exactly two in-arcs, we may assume $t_{1}u_{1}\in A(T_{1}')$ (that is, $t_{1}\in V(T_{1}')$) and $s_{2}u_{1}\in A(T_{2}')$ (that is, $s_{2}\in V(T_{2}')$). Since $T_{1}'$ and $T_{2}'$ are pendant $(S,r)$-trees and $r$ has exactly two out-arcs now, we have $s_{1}\in V(T_{1}')$. Thus, there exists a path from $s_{1}$ to $t_{1}$, namely $P_{1}'$. Analogously, we have $t_{2}\in V(T_{2}')$, and there exists a path from $s_{2}$ to $t_{2}$, namely $P_{2}'$. Observe that $P_{1}'$ and $P_{2}'$ are vertex disjoint as $T_{1}'$ and $T_{2}'$ are internally-disjoint. This implies that $[D^{*};s_{1},s_{2},t_{1},t_{2}]$ is a positive instance of Directed 2-Linkage on Eulerian digraphs, as desired.
\end{proof}


\subsection{Symmetric digraphs}

We turn our attention to symmetric digraphs now. Recall that a connected symmetric digraph is Eulerian. However, unlike Eulerian digraphs, the problem of deciding whether $\tau_{S,r}(D)\geq \ell$ on symmetric digraphs is polynomial-time solvable, if $|S|=k\geq 3$ and $\ell\geq 2$ are both fixed integers. Before presenting the proof, we first need the following results by Sun and Zhang.

\begin{lemma}\cite{Sun-Zhang}\label{lem2.1}
Let $D$ be a symmetric digraph with $S \subseteq V(D)$. Let $\{s_i, t_i\mid i\in [p]\}\subseteq V(D)$. The following problem can be solved in time $f(p)|V(D)|^2$ for some function $f(p)$: decide whether there is an $(s_i,t_i)$-path, say $P_i$, for each $i\in [p]$, such that none of the internal vertices of $P_i$ are contained in  $S\cup \bigcup_{j\in [p]\setminus\{i\}}V(P_j)$.
\end{lemma}

Now, by Lemma \ref{lem2.1}, we give the following Theorem which shows the polynomiality for $\tau_{S,r}(D)$ on symmetric digraphs in Table~\ref{tab:symmetric_digraphs}. Note that $O(n^{\ell(k-2)+2}\cdot k^{\ell(2k-3)})$ is a polynomial in $n$ when both $k$ and $\ell$ are fixed integers.

\begin{theorem}\label{thm2.3}
    Let $D$ be a symmetric digraph and $S\subseteq V(D)$ with $|S|=k$ and $r\in S$. Let $k\geq 3$ and $\ell\geq 2$ be fixed integers. The problem of deciding if $\tau_{S,r}(D)\geq \ell$ can be solved in time $O(n^{\ell(k-2)+2}\cdot k^{\ell(2k-3)})$.
\end{theorem}

\begin{proof}
    Without loss of generality, we may assume that all pendant $(S,r)$-trees are minimal (that is, the set of vertices with degree one is exactly $S$). Let $T$ be any pendant $(S,r)$-tree in $D$. Let $R$ be all vertices in $T$ with degree at least three. Observe that $|R|\leq |S|-2$, and $d_{T}^{+}(x)=d_{T}^{-}(x)=1$ holds for every vertex $x\in V(T)\backslash (R\cup S)$. Let $T^{s}$ be the {\em skeleton} of $T$, which is obtained from $T$ by contracting all vertices in $V(T)\backslash (R\cup S)$. Note that $V(T^{s})=R\cup S$. Thus, there are at most $n^{k-2}$ possibilities for $V(T^{s})$ and less than $k^{2k-3}$ possibilities for the arcs of $T^{s}$, and so there are at most $n^{k-2}\cdot k^{2k-3}$ different skeletons of pendant $(S,r)$-trees in $D$.

    Our algorithm will try all possible $\ell$-tuples, $\mathcal{T}^{s}=(T_{1}^{s},T_{2}^{s},\cdots,T_{\ell}^{s})$, of skeletons of pendant $(S,r)$-trees and determine if there is a set of $\ell$ pairwise internally-disjoint pendant $(S,r)$-trees,  say $\{T_i\mid i\in [\ell]\}$, such that $T_{i}^{s}$ is the skeleton of $T_{i}$. If such a set of trees exists for a certain $\mathcal{T}^{s}$, then we return this solution, otherwise, we return that no solution exists. We will first prove that this algorithm gives the correct answer and then compute its time complexity.

    If our algorithm returns a solution, then clearly a solution exists. So now assume that a solution exists and let $\{T_i\mid i\in [\ell]\}$ be the desired set of pairwise internally-disjoint pendant $(S,r)$-trees. When we consider $\mathcal{T}^{s}=(T_{1}^{s},T_{2}^{s},\cdots,T_{\ell}^{s})$, where $T_{i}^{s}$ is the skeleton of $T_{i}$, our algorithm will find a solution, so the algorithm always returns a solution if one exists.

    Given such an $\ell$-tuples, $\mathcal{T}^{s}$, we need to determine if there exists a set of $\ell$ pairwise internally-disjoint pendant $(S,r)$-trees, $\{T_i\mid i\in [\ell]\}$, such that $T_{i}^{s}$ is the skeleton of $T_{i}$. We first check that no vertices in $V(D)\backslash S$ belongs to more than one skeleton. If the above does not hold then the desired trees do not exist, so assume the above holds. By Lemma \ref{lem2.1}, for every arc $uv\notin A(D)$ that belongs to some skeleton $T_{i}^{s}$, we can find a $(u,v)$-path in $D-A(D[S])$, such that no internal vertex on any path belongs to $S$ or to a different path, and this can be done in $O(n^{2})$ time. If such paths exist, then we can obtain the desired pendant $(S,r)$-trees by substituting the arcs $uv$ by the $(u,v)$-paths. Otherwise, the desired pendant $(S,r)$-trees do not exist. 
    
    Now, we analyze the time complexity. The number of different $\ell$-tuples, $\mathcal{T}^{s}$, that we need to consider is bounded by the function $(n^{k-2}\cdot k^{2k-3})^{\ell}$. Therefore, the algorithm has time complexity $O(n^{\ell(k-2)+2}\cdot k^{\ell(2k-3)})$.
\end{proof}


We now consider the NP-complete cases in Table~\ref{tab:symmetric_digraphs}. Chen, Li, Liu and Mao \cite{Chen-Li-Liu-Mao} introduced the following problem and proved that it is NP-complete.

\textbf{CLLM problem}: Given a tripartite graph $G=(V, E)$ with a
3-partition $(A,B,C)$ such that $|A|=|B|=|C|=q$, decide whether
there is a partition of $V$ into $q$ disjoint 3-sets $V_1, \dots,
V_q$ such that for every $V_i= \{a_{i_1}, b_{i_2}, c_{i_3}\}$, we have
$a_{i_1} \in A, b_{i_2} \in B, c_{i_3} \in C$, and the induced subgraph $G[V_i]$ of $G$ is connected.


\begin{lemma}\label{lem2.2}\cite{Chen-Li-Liu-Mao}
    The CLLM problem is NP-complete.
\end{lemma}

By Lemma~\ref{lem2.2}, we obtain the following result on symmetric digraphs. Note that we use the construction from the proof of Theorem 3.6 in \cite{Sun-Yeo}.


\begin{theorem}\label{thm2.4}
    Let $D$ be a symmetric digraph and $S\subseteq V(D)$ with $|S|=k$ and $r\in S$. Let $k\geq 3$ be a fixed integer. The problem of deciding if $\tau_{S,r}(D)\geq \ell$ ($\ell\geq 2$ is part of the input) is NP-complete.
\end{theorem}

\begin{proof}
    Let $G$ be a tripartite graph with 3-partition $(A,B,C)$ such that $|A|=|B|=|C|=\ell$. We now construct a new digraph $D$ (from $G$) with
    \begin{align*}
        V(D)=&V(G)\cup S'\cup \{r\},
    \end{align*}
    where $S'=\{s_{i}\mid i\in [k-1]\}$. The arc-set of $D$ is defined as
    \begin{align*}
        A(D)=&\{xy,yx\mid xy\in E(G)\}\\
             &\cup\{ru,ur,s_{1}v,vs_{1}\mid u\in A~\mbox{and}~v\in B\}\\
             &\cup\{s_{i}w,ws_{i}\mid i=2,3,\cdots,k-1~\mbox{and}~w\in C\}.
    \end{align*}

    \begin{figure}
        \centering
        \begin{tikzpicture}
            \node[circle,draw,fill=white,inner sep=5pt,line width=0.8pt] (r) at (0,1.5) {};
            \draw (-2,-0.5) rectangle (2,0.5);
            \node[circle,draw,fill=white,inner sep=5pt,line width=0.8pt] (u1) at (-1.5,0) {};
            \node[circle,draw,fill=white,inner sep=5pt,line width=0.8pt] (u2) at (-0.8,0) {};
            \node at (-0.2,0) {$\cdots$};
            \node[circle,draw,fill=white,inner sep=5pt,line width=0.8pt] (ui) at (0.4,0) {};
            \node at (1.0,0) {$\cdots$};
            \node[circle,draw,fill=white,inner sep=5pt,line width=0.8pt] (uq) at (1.6,0) {};
            \node at (-2.5,0) {$A$};
        
            \node[circle,draw,fill=white,inner sep=5pt,line width=0.8pt] (s1) at (-2.5,-4) {};
            \draw (-4.5,-3) rectangle (-0.5,-2);
            \node[circle,draw,fill=white,inner sep=5pt,line width=0.8pt] (v1) at (-4,-2.5) {};
            \node[circle,draw,fill=white,inner sep=5pt,line width=0.8pt] (v2) at (-3.3,-2.5) {};
            \node at (-2.7,-2.5) {$\cdots$};
            \node[circle,draw,fill=white,inner sep=5pt,line width=0.8pt] (vj) at (-2.1,-2.5) {};
            \node at (-1.5,-2.5) {$\cdots$};
            \node[circle,draw,fill=white,inner sep=5pt,line width=0.8pt] (vq) at (-0.9,-2.5) {};
            \node at (-5,-2.5) {$B$};

            \draw (0.5,-3) rectangle (4.5,-2);
            \node[circle,draw,fill=white,inner sep=5pt,line width=0.8pt] (w1) at (1,-2.5) {};
            \node[circle,draw,fill=white,inner sep=5pt,line width=0.8pt] (w2) at (1.7,-2.5) {};
            \node at (2.3,-2.5) {$\cdots$};
            \node[circle,draw,fill=white,inner sep=5pt,line width=0.8pt] (wt) at (2.9,-2.5) {};
            \node at (3.5,-2.5) {$\cdots$};
            \node[circle,draw,fill=white,inner sep=5pt,line width=0.8pt] (wq) at (4.1,-2.5) {};
            \node at (5,-2.5) {$C$};

            \draw (1,-5) rectangle (4,-4);
            \node[circle,draw,fill=white,inner sep=5pt,line width=0.8pt] (s2) at (1.5,-4.5) {};
            \node[circle,draw,fill=white,inner sep=5pt,line width=0.8pt] (s3) at (2.2,-4.5) {};
            \node at (2.8,-4.5) {$\cdots$};
            \node[circle,draw,fill=white,inner sep=7pt,line width=0.8pt] (sk-1) at (3.5,-4.5) {};

            \draw[dash pattern=on 5pt off 3pt] (-5.25,-3.25) rectangle (5.25,0.75);
            \node at (4.3,0.3) {$\overleftrightarrow{G}$};
            \draw (-5.5,-5.5) rectangle (5.5,2);
            \node at (-4.5,-5) {$D$};

            \draw[<->] (r) -- (0,0.5);
            \draw[<->] (s1) -- (-2.5,-3);
            \draw[<->] (2.5,-4) -- (2.5,-3);
            \draw[<->] (ui) to [bend right=20] (vj);
            \draw[<->] (ui) to [bend left=20] (wt);
            \draw[<->] (vj) to [bend left=30] (wt);

            \node at (r) {$r$};
            \node at (s1) {$s_{1}$};
            \node at (s2) {$s_{2}$};
            \node at (s3) {$s_{3}$};
            \node at (sk-1) {$s_{k-1}$};
            \node at (u1) {$u_{1}$};
            \node at (u2) {$u_{2}$};
            \node at (ui) {$u_{i}$};
            \node at (uq) {$u_{\ell}$};
            \node at (v1) {$v_{1}$};
            \node at (v2) {$v_{2}$};
            \node at (vj) {$v_{j}$};
            \node at (vq) {$v_{\ell}$};
            \node at (w1) {$w_{1}$};
            \node at (w2) {$w_{2}$};
            \node at (wt) {$w_{t}$};
            \node at (wq) {$w_{\ell}$};
            \node [right] at (-3.6,-1.15) {$u_{i}v_{j}\in E(G)$};
            \node [right] at (2.1,-1.15) {$u_{i}w_{t}\in E(G)$};
            \node [right] at (-0.8,-1.4) {$v_{j}w_{t}\in E(G)$};
        \end{tikzpicture}
        \caption{The symmetric digraph $D$}
        \label{fig:thm2.4}
    \end{figure}

    This completes the construction of $D$ as shown in Figure \ref{fig:thm2.4}. Clearly, $D$ is symmetric. Let $S=\{r\}\cup S'$ and let $r$ be the root. We will show that $\tau_{S,r}(D)\geq \ell$ if and only if $G$ is a positive instance of the CLLM problem. 

First assume that $G$ is a positive instance of the CLLM problem, that is, there is a partition of $V(G)$ into $\ell$ disjoint 3-sets $V_1, \dots, V_{\ell}$ such that for every $V_i= \{a_{i_1}, b_{i_2}, c_{i_3}\}$, we have $a_{i_1} \in A, b_{i_2} \in B, c_{i_3} \in C$, and the induced subgraph $G[V_i]$ of $G$ is connected. For each $i\in [\ell]$, we construct a pendant $(S,r)$-tree $T_{i}$ which is the union of the arcs $ra_{i_{1}}, b_{i_{2}}s_{1}$, all arcs from $c_{i_{3}}$ to $S'\backslash \{s_{1}\}$, and a $(V_i, a_{i_1})$-tree in $D[V_i]$ (note that such a tree must exist since $G[V_{i}]$ is connected). It can be checked that $\{T_i\mid i\in [\ell]\}$ is a set of pairwise internally-disjoint pendant $(S,r)$-trees in $D$, as desired.


    Conversely, assume that $\tau_{S,r}(D)\geq \ell$ and let $\{T'_i\mid i\in [\ell]\}$ be a set of pairwise internally-disjoint pendant $(S,r)$-trees in $D$. By the construction of $A(D)$ and the fact that $|A|=|B|=|C|=\ell$, we have $$|V(T'_{i})\cap A|=|V(T'_{i})\cap B|=|V(T'_{i})\cap C|=1$$ for each $i\in [\ell]$. Moreover, let $$V(T'_{i})\cap A=\{a_{i}\}, V(T'_{i})\cap B=\{b_{i}\}, V(T'_{i})\cap C=\{c_{i}\}.$$ Observe that $G[a_{i},b_{i},c_{i}]$ is connected for each $i\in [\ell]$. Thus, $G$ is a positive instance of the CLLM problem. By the above argument and Lemma~\ref{lem2.2}, we complete the proof.
\end{proof}

In order to prove the case that $k$ is part of input and $\ell\geq 2$ is a constant, we need the complexity of the following problem.

\textbf{2-Coloring Hypergraph}: Let $H$ be a hypergraph with vertex set $V(H)$ and edge set $E(H)$, determine whether there exists a 2-coloring of $V(H)$ such that each hyperedge of $E(H)$ contains vertices of both colors.

Lovász proved the NP-hardness of 2-Coloring Hypergraph.

\begin{theorem}\label{thm2.5}\cite{Lovász}
    The problem of 2-Coloring Hypergraph is NP-hard.
\end{theorem}

Theorem \ref{thm2.4} and the following theorem imply all the NP-completeness entries in Table \ref{tab:symmetric_digraphs}. 

\begin{theorem}\label{thm2.6}
    Let $D$ be a symmetric digraph and $S\subseteq V(D)$ with $|S|=k$ ($k$ is part of the input) and $r\in S$. Let $\ell\geq 2$ be a fixed integer. The problem of deciding if $\tau_{S,r}(D)\geq \ell$ is NP-complete.
\end{theorem}

\begin{proof}
    Let $H$ be a hypergraph with vertex set $V(H)$ and edge set $E(H)$. We construct a new digraph $D$ (from $H$) with $$V(D)=V(H)\cup E(H)\cup U\cup \{r\},$$ where $U=\{u_i\mid i\in [\ell-2]\}$. The arc-set of $D$ is defined as 
    \begin{align*}
        A(D)=&\{xe,ex\mid x\in V(H),e\in E(H)~\mbox{and}~x\in V(e) \}\\
             &\cup\{ru_{i},u_{i}r,u_{i}e,eu_{i}\mid u_{i}\in U~\mbox{and}~e\in E(H)\}\\
             &\cup\{xy,yx,rz,zr\mid x,y,z\in V(H)~\mbox{and}~x\neq y\}.
    \end{align*}

    \begin{figure}
        \centering
        \begin{tikzpicture}
            \draw (-1,-0.5) rectangle (4,1.2);
            \node[circle,draw,fill=white,inner sep=5pt,line width=0.8pt] (v1) at (-0.5,0) {};
            \node[circle,draw,fill=white,inner sep=5pt,line width=0.8pt] (v2) at (0.5,0) {};
            \node at (1.25,0) {$\cdots$};
            \node[circle,draw,fill=white,inner sep=5pt,line width=0.8pt] (vi) at (2,0) {};
            \node at (2.75,0) {$\cdots$};
            \node[circle,draw,fill=white,inner sep=5pt,line width=0.8pt] (vn) at (3.5,0) {};
            \node at (5.5,0) {$V(H)$};
        
            \draw (-1,-2.2) rectangle (4,-1.2);
            \node[circle,draw,fill=white,inner sep=5pt,line width=0.8pt] (ej) at (2,-1.7) {};
            \node at (5.5,-1.7) {$E(H)$};

            \draw (-1,-4) rectangle (4,-3);
            \node[circle,draw,fill=white,inner sep=5pt,line width=0.8pt] (u1) at (0,-3.5) {};
            \node[circle,draw,fill=white,inner sep=5pt,line width=0.8pt] (u2) at (0.7,-3.5) {};
            \node[circle,draw,fill=white,inner sep=7pt,line width=0.8pt] (u3) at (2.5,-3.5) {};
            \node at (1.65,-3.5) {$\cdots$};
            \node at (5.5,-3.5) {$U$};
            \node[circle,draw,fill=white,inner sep=5pt,line width=0.8pt] (r) at (-2,-3.5) {};

            \draw (-3,-4.5) rectangle (6.5,1.5);
            \node at (-2,0.5) {$D$};

            \draw[<->] (v1) -- (v2);
            \draw[<->] (v1) to [bend left=30] (vi);
            \draw[<->] (v1) to [bend left=50] (vn);
            \draw[<->] (v2) to [bend right=15] (vi);
            \draw[<->] (v2) to [bend left=30] (vn);
            \draw[<->] (vi) to [bend right=30] (vn);
            \draw[<->] (1.5,-2.2) -- (1.5,-3);
            \draw[<->] (r) -- (-1,-3.5);
            \draw[<->] (r) to [bend left=45] (-1,0);           
            \draw[<->] (vi) -- (ej);

            \node at (v1) {$v_{1}$};
            \node at (v2) {$v_{2}$};
            \node at (vi) {$v_{i}$};
            \node at (vn) {$v_{n}$};
            \node at (ej) {$e_{j}$};
            \node at (u1) {$u_{1}$};
            \node at (u2) {$u_{2}$};
            \node at (u3) {$u_{\ell-2}$};
            \node at (r) {$r$};
            \node [right] at (1.92,-0.88) {$v_{i}\in e_{j}$};
        \end{tikzpicture}
        \caption{The symmetric digraph $D$}
        \label{fig:thm2.6}
    \end{figure}
    
    This completes the construction of $D$ as shown in Figure~\ref{fig:thm2.6}. Clearly, $D$ is symmetric. Let $S=E(H)\cup\{r\}$ and let $r$ be the root. We will show that $\tau_{S,r}(D)\geq \ell$ if and only if $H$ is 2-colorable.

    First assume that $H$ is 2-colorable and let $R$ (resp. $B$) be the set of red vertices (resp. blue vertices) of $V(H)$ in a proper 2-coloring of $H$. Let $r_{1}\in R$ and $b_{1}\in B$. For each $i\in [\ell-2]$, let $T_{i}$ be a pendant $(S,r)$-tree which is the union of the arc $ru_{i}$ and all arcs from $u_{i}$ to $S\backslash \{r\}$. Let $T_{\ell-1}$ be a pendant $(S,r)$-tree which contains the arc $rr_{1}$, all arcs from $r_{1}$ to $R\backslash \{r_{1}\}$, and an arc from some vertex of $R$ to $e$ for each $e\in E(H)$ (such an arc exists for each $e\in E(H)$, as every edge in $E(H)$ contains a red vertex).
    Analogously, let $T_{\ell}$ be a pendant $(S,r)$-tree which contains the arc $rb_{1}$, all arcs from $b_{1}$ to $B\backslash \{b_{1}\}$, and an arc from some vertex of $B$ to $e$ for each $e\in E(H)$ (such an arc exists for each $e\in E(H)$, as every edge in $E(H)$ contains a blue vertex). It can be checked that $\{T_i\mid i\in [\ell]\}$ is a set of $\ell$ pairwise internally-disjoint pendant $(S,r)$-trees, which means $\tau_{S,r}(D)\geq \ell$.


Conversely, assume that $\tau_{S,r}(D)\geq \ell$ and let $\{T'_i\mid i\in [\ell]\}$ be a set of $\ell$ pairwise internally-disjoint pendant $(S,r)$-trees in $D$. As $|U| = \ell-2$, there are at least two of these trees, say $T_1'$ and $T_2'$, which contain no vertex from $U$. Let $R'$ (resp. $B'$) be all in-neighbours of $E(H)$ in $T_1'$ (resp. $T_2'$), that is, $R' = N_{T_1'}^-(E(H))$ and $B' = N_{T_2'}^-(E(H))$ (moreover, $R'\cap B'=\emptyset$, as $T_1'$ and $T_2'$ are internally-disjoint). Since each $e\in E(H)$ has an in-arc in $T'_1$  (resp. $T'_2$), $e$ must contain at least a vertex from $R'$ (resp. $B'$). Finally, any vertex in $V(H)\setminus (R'\cup B')$ can be assigned randomly to either $R'$ or $B'$. In this way, we get a proper $2$-coloring of $H$, that is, $H$ is $2$-colorable. This completes the proof.

\end{proof}

\subsection{Hardness of approximating IDPSTP}
In this subsection, we show that, unless $P=NP$, any approximation algorithm for the IDPSTP problem has a guarantee of $\Omega(n^{{1/3}-\epsilon})$. This result is derived by constructing a reduction from the Directed 2-Linkage problem, where the reduction is inspired by that used in \cite{Che-Sal} for the internally-disjoint directed Steiner tree problem. 

\begin{theorem}\label{inappro}
    For any given instance of IDPSTP, it is NP-hard to approximate the solution within $O(n^{{1/3}-\epsilon})$ for any $\epsilon> 0$.
\end{theorem}

\begin{proof}
    We will construct a gap amplifier from Directed 2-Linkage to IDPSTP.
    
    Let $I=[H;x_{1},y_{1},x_{2},y_{2}]$ be an instance of Directed 2-Linkage which is shown in Figure \ref{fig:digraph H}(a). If the instance is positive, then there exist two vertex-disjoint directed paths $P_{1}, P_{2}$, where $P_{i}$ starts at $x_{i}$ and ends at $y_{i}$ for all $i\in[2]$.
    \begin{figure}[h]
        \centering
        \begin{tikzpicture}
            \begin{scope} [xshift=0cm]
            \draw (0,0) circle (1.5);
            \coordinate (x1) at (-1.5,0);
            \coordinate (x2) at (0,1.5);
            \coordinate (y1) at (1.5,0);
            \coordinate (y2) at (0,-1.5);
    
            \node[circle,fill,inner sep=2pt] at (x1) {};
            \node[circle,fill,inner sep=2pt] at (x2) {};
            \node[circle,fill,inner sep=2pt] at (y1) {};
            \node[circle,fill,inner sep=2pt] at (y2) {};
    
            \node[left] at (x1) {$x_{1}$};
            \node[above] at (x2) {$x_{2}$};
            \node[right] at (y1) {$y_{1}$};
            \node[below] at (y2) {$y_{2}$};
    
            \node at (0,0) {\Large$H$};
            \node at (0,-2.5) {(a) The graph $H$};
            \end{scope}

            \begin{scope} [xshift=6cm]
            \draw (0,0) circle (1.5);
            \coordinate (x1) at (-1.5,0);
            \coordinate (x2) at (0,1.5);
            \coordinate (y1) at (1.5,0);
            \coordinate (y2) at (0,-1.5);
    
            \node[circle,fill,inner sep=2pt] at (x1) {};
            \node[circle,fill,inner sep=2pt] at (x2) {};
            \node[circle,fill,inner sep=2pt] at (y1) {};
            \node[circle,fill,inner sep=2pt] at (y2) {};
    
            \node[left] at (x1) {$x_{1,ij}^{k}$};
            \node[above] at (x2) {$x_{2,ij}^{k}$};
            \node[right] at (y1) {$y_{1,ij}^{k}$};
            \node[below] at (y2) {$y_{2,ij}^{k}$};
    
            \node at (0,0) {\Large$H_{ij}^{k}$};
            \node at (0,-2.5) {(b) The copy of $H$};
            \end{scope}
        \end{tikzpicture}
        \caption{An instance of Directed 2-Linkage}
        \label{fig:digraph H}
    \end{figure}
    
    Given an $\epsilon> 0$ and let $N=|V(H)|^{1/\epsilon}$. We construct a new digraph $D'$ (from $H$) with the following operations: Firstly, we create $\frac{(N^{2}-2N+2)(N-1)}{2}$ copies of the digraph $H$, and place all copies of $H$ at the position marked with $H$ in Figure \ref{fig:digraph D'}. Secondly, we create two sets of vertices $S_{\alpha}=\{s_{i}\mid i\in[N]\},S_{\beta}=\{t_{i}\mid i\in[N]\}$ and let $S_{i}=\{s_{i}\}\cup S_{\beta}\backslash \{t_{i}\}$ for each $i\in[N]$. Meanwhile, we add $N(N-1)$ additional vertices $u_{i,j}$ where $i,j\in[N]$ and $i\neq j$, and arrange these vertices in the positions indicated in Figure \ref{fig:digraph D'}. Thirdly, let all the edges in Figure \ref{fig:digraph D'} (except for the both vertices of arcs inside each $H$) be oriented from top to bottom and from left to right.

\tikzset{
    vertex/.style={circle,draw,fill=white,inner sep=1.5pt,line width=0.8pt},
    solidvertex/.style = {circle,fill = black,minimum size = 5pt,inner sep = 0pt}, 
    base vertex/.style={circle,draw,fill=white,inner sep=3.8pt,line width=0.4pt
    }
}

\begin{figure}[H]
    \centering
    \begin{tikzpicture}
        \node [solidvertex] (b1) at (1.3,-1.5) {};
        \node [solidvertex] (b2) at (3.85,-1.5) {};
        \node [solidvertex] (b3) at (6.5,-1.5) {};
        \node at (7.1,-1.5) {$\cdots$};
        \node at (7.6,-1.5) {$\cdots$};
        \node at (8.1,-1.5) {$\cdots$};
        \node at (8.6,-1.5) {$\cdots$};
        \node at (9.1,-1.5) {$\cdots$};
        \node [solidvertex] (b4) at (9.65,-1.5) {};
        \node [solidvertex] (b5) at (11.7,-1.5) {};

        \node [solidvertex] (a1) at (0,0) {};
        \node [base vertex] (G1) at (0.4,0) {};
        \node [base vertex] (G2) at (0.9,0) {};
        \node [base vertex] (G3) at (1.85,0) {};
        \node [base vertex] (G4) at (2.35,0) {};

        \node [vertex] (x11) at (2.7,0) {};
        \node [base vertex] (G5) at (3.05,0) {};
        \node [base vertex] (G6) at (3.55,0) {};
        \node [base vertex] (G7) at (4.5,0) {};
        \node [base vertex] (G8) at (5,0) {};

        \node [vertex] (x12) at (5.35,0) {};
        \node [base vertex] (G9) at (5.7,0) {};
        \node [base vertex] (G10) at (6.2,0) {};
        \node [base vertex] (G11) at (7.15,0) {};
        \node [base vertex] (G12) at (7.65,0) {};

        \node [vertex] (x13) at (8.5,0) {};
        \node [base vertex] (G13) at (8.85,0) {};
        \node [base vertex] (G14) at (9.35,0) {};
        \node [base vertex] (G15) at (10.3,0) {};
        \node [base vertex] (G16) at (10.8,0) {};

        \node [vertex] (x14) at (11.15,0) {};

        \node [solidvertex] (a2) at (0,1) {};
        \node [vertex] (x21) at (0.4,1) {};
        \node [base vertex] (G17) at (0.9,1) {};
        \node [base vertex] (G18) at (1.85,1) {};
        \node [base vertex] (G19) at (2.35,1) {};

        \node [base vertex] (G20) at (3.05,1) {};
        \node [base vertex] (G21) at (3.55,1) {};
        \node [base vertex] (G22) at (4.5,1) {};
        \node [base vertex] (G23) at (5,1) {};

        \node [vertex] (x22) at (5.7,1) {};
        \node [base vertex] (G24) at (6.2,1) {};
        \node [base vertex] (G25) at (7.15,1) {};
        \node [base vertex] (G26) at (7.65,1) {};

        \node [vertex] (x23) at (8.85,1) {};
        \node [base vertex] (G27) at (9.35,1) {};
        \node [base vertex] (G28) at (10.3,1) {};
        \node [base vertex] (G29) at (10.8,1) {};

        \node [vertex] (x24) at (11.15,1) {};

        \node [solidvertex] (a3) at (0,2) {};
        \node [vertex] (x31) at (0.9,2) {};
        \node [base vertex] (G30) at (1.85,2) {};
        \node [base vertex] (G31) at (2.35,2) {};

        \node [vertex] (x32) at (3.05,2) {};
        \node [base vertex] (G32) at (3.55,2) {};
        \node [base vertex] (G33) at (4.5,2) {};
        \node [base vertex] (G34) at (5,2) {};

        \node [base vertex] (G35) at (6.2,2) {};
        \node [base vertex] (G36) at (7.15,2) {};
        \node [base vertex] (G37) at (7.65,2) {};

        \node [vertex] (x33) at (9.35,2) {};
        \node [base vertex] (G38) at (10.3,2) {};
        \node [base vertex] (G39) at (10.8,2) {};

        \node [vertex] (x34) at (11.15,2) {};

        \node at (0,2.5) {$\vdots$};
        \node at (0,3) {$\vdots$};
        \node at (0,3.5) {$\vdots$};
        \node at (0,4) {$\vdots$};
        \node at (0,4.5) {$\vdots$};
        \node [solidvertex] (a4) at (0,5) {};
        \node [vertex] (x41) at (1.85,5) {};
        \node [base vertex] (G40) at (2.35,5) {};

        \node [vertex] (x42) at (4.5,5) {};
        \node [base vertex] (G41) at (5,5) {};

        \node [vertex] (x43) at (7.15,5) {};
        \node [base vertex] (G42) at (7.65,5) {};

        \node [base vertex] (G43) at (10.8,5) {};

        \node [vertex] (x44) at (11.15,5) {};

        \node [solidvertex] (a5) at (0,6) {};
        \node [vertex] (x51) at (2.35,6) {};
        \node [vertex] (x52) at (5,6) {};
        \node [vertex] (x53) at (7.65,6) {};
        \node [vertex] (x54) at (10.8,6) {};

        \draw (a1) -- (G1);
        \draw (G1) -- (G2);
        \draw [dashed] (G2) -- (G3);
        \draw (G3) -- (G4);
        \draw (G4) -- (x11);
        \draw (x11) -- (G5);
        \draw (G5) -- (G6);
        \draw [dashed] (G6) -- (G7);
        \draw (G7) -- (G8);
        \draw (G8) -- (x12);
        \draw (x12) -- (G9);
        \draw (G9) -- (G10);
        \draw [dashed] (G10) -- (G11);
        \draw (G11) -- (G12);
        \draw [dashed] (G12) -- (x13);
        \draw (x13) -- (G13);
        \draw (G13) -- (G14);
        \draw [dashed] (G14) -- (G15);
        \draw (G15) -- (G16);
        \draw (G16) -- (x14);

        \draw (a2) -- (x21);
        \draw (x21) -- (G17);
        \draw [dashed] (G17) -- (G18);
        \draw (G18) -- (G19);
        \draw (G19) -- (G20);
        \draw (G20) -- (G21);
        \draw [dashed] (G21) -- (G22);
        \draw (G22) -- (G23);
        \draw (G23) -- (x22);
        \draw (x22) -- (G24);
        \draw [dashed] (G24) -- (G25);
        \draw (G25) -- (G26);
        \draw [dashed] (G26) -- (x23);
        \draw (x23) -- (G27);
        \draw [dashed] (G27) -- (G28);
        \draw (G28) -- (G29);
        \draw (G29) -- (x24);

        \draw (a3) -- (x31);
        \draw [dashed] (x31) -- (G30);
        \draw (G30) -- (G31);
        \draw (G31) -- (x32);
        \draw (x32) -- (G32);
        \draw [dashed] (G32) -- (G33);
        \draw (G33) -- (G34);
        \draw (G34) -- (G35);
        \draw [dashed] (G35) -- (G36);
        \draw (G36) -- (G37);
        \draw [dashed] (G37) -- (x33);
        \draw [dashed] (x33) -- (G38);
        \draw (G38) -- (G39);
        \draw (G39) -- (x34);

        \draw (a4) -- (x41);
        \draw (x41) -- (G40);
        \draw (G40) -- (x42);
        \draw (x42) -- (G41);
        \draw (G41) -- (x43);
        \draw (x43) -- (G42);
        \draw [dashed] (G42) -- (G43);
        \draw (G43) -- (x44);

        \draw (a5) -- (x51);
        \draw (x51) -- (x52);
        \draw (x52) -- (x53);
        \draw [dashed] (x53) -- (x54);

        \draw (x51) -- (G40);
        \draw (x52) -- (G41);
        \draw (x53) -- (G42);
        \draw (x54) -- (G43);
        
        \draw [dashed] (x41) -- (G30);
        \draw [dashed] (G40) -- (G31);
        \draw [dashed] (x42) -- (G33);
        \draw [dashed] (G41) -- (G34);
        \draw [dashed] (x43) -- (G36);
        \draw [dashed] (G42) -- (G37);
        \draw [dashed] (G43) -- (G39);

        \draw (x31) -- (G17);
        \draw (G30) -- (G18);
        \draw (G31) -- (G19);
        \draw (x32) -- (G20);
        \draw (G32) -- (G21);
        \draw (G33) -- (G22);
        \draw (G34) -- (G23);
        \draw (G35) -- (G24);
        \draw (G36) -- (G25);
        \draw (G37) -- (G26);
        \draw (x33) -- (G27);
        \draw (G38) -- (G28);
        \draw (G39) -- (G29);

        \draw (x21) -- (G1);
        \draw (G17) -- (G2);
        \draw (G18) -- (G3);
        \draw (G19) -- (G4);
        \draw (G20) -- (G5);
        \draw (G21) -- (G6);
        \draw (G22) -- (G7);
        \draw (G23) -- (G8);
        \draw (x22) -- (G9);
        \draw (G24) -- (G10);
        \draw (G25) -- (G11);
        \draw (G26) -- (G12);
        \draw (x23) -- (G13);
        \draw (G27) -- (G14);
        \draw (G28) -- (G15);
        \draw (G29) -- (G16);

        \foreach \i in {1,...,4}{
            \draw (G\i) -- (b1);
            \draw (G\the\numexpr\i+4\relax) -- (b2);
            \draw (G\the\numexpr\i+8\relax) -- (b3);
            \draw (G\the\numexpr\i+12\relax) -- (b4); 
        }
        \draw (x11) -- (b2);
        \draw (x12) -- (b3);
        \draw (x13) -- (b4);

        \draw (x14) -- (b5);
        \draw (x24) -- (b5);
        \draw (x34) -- (b5);
        \draw (x44) -- (b5);
        \draw (x54) -- (x44);
        \foreach \i in {1,...,3}{
            \node[below] at (a\i) {$s_{\i}$};
            \node[below] at (b\i) {$t_{\i}$};
        }
        \node [below] at (a4) {$s_{N-1}$};
        \node [below] at (a5) {$s_{N}$};
        \node [below] at (b4) {$t_{N-1}$};
        \node [below] at (b5) {$t_{N}$};
        \node [above,scale=0.9] at (x11) {$u_{1,2}$};
        \node [above,scale=0.9] at (x12) {$u_{1,3}$};
        \node [above,scale=0.9] at (x13) {$u_{1,N-1}$};
        \node [above,scale=0.9] at (x14) {$u_{1,N}$};

        \node [above,scale=0.9] at (x21) {$u_{2,1}$};
        \node [above,scale=0.9] at (x22) {$u_{2,3}$};
        \node [above,scale=0.9] at (x23) {$u_{2,N-1}$};
        \node [above,scale=0.9] at (x24) {$u_{2,N}$};

        \node [above,scale=0.9] at (x31) {$u_{3,1}$};
        \node [above,scale=0.9] at (x32) {$u_{3,2}$};
        \node [above,scale=0.9] at (x33) {$u_{3,N-1}$};
        \node [above,scale=0.9] at (x34) {$u_{3,N}$};

        \node [above,scale=0.9] at (x41) {$u_{N-1,1}$};
        \node [above,scale=0.9] at (x42) {$u_{N-1,2}$};
        \node [above,scale=0.9] at (x43) {$u_{N-1,3}$};
        \node [above,scale=0.9] at (x44) {$u_{N-1,N}$};

        \node [above,scale=0.9] at (x51) {$u_{N,1}$};
        \node [above,scale=0.9] at (x52) {$u_{N,2}$};
        \node [above,scale=0.9] at (x53) {$u_{N,3}$};
        \node [above,scale=0.9] at (x54) {$u_{N,N-1}$};
        \foreach \i in {1,...,43} {
            \node at (G\i) {$H$};
        }

        \draw[dotted, thick, green] (8.65,-0.35) rectangle (11.05,0.35);
        \draw[dotted, thick, red] (10.55,-0.3) rectangle (11.03,5.35);

    \end{tikzpicture}
    \caption{The construction of $D'$}
    \label{fig:digraph D'}
\end{figure}
    
    Let $1\leq i<k\leq N$, $j\in [N-1]$ and $j\neq k$. For each copy of $H$, we denote as $H_{ij}^{k}$ according to its position. Specifically, $H_{ij}^{k}$ is located at the intersection of the horizontal layer where vertex $s_{i}$ lies and the vertical layer where vertex $u_{k,j}$ lies. Meanwhile, we relabel the four vertices $x_{1},y_{1},x_{2},y_{2}$ in the instance of $H$ as $x_{1,ij}^{k},y_{1,ij}^{k},x_{2,ij}^{k},y_{2,ij}^{k}$ in $H_{ij}^{k}$ (see Figure \ref{fig:digraph H}(b)), and denote the paths $P_{1},P_{2}$ in $H$ by $P_{1,ij}^{k},P_{2,ij}^{k}$ in $H_{ij}^{k}$ respectively. Note that, for any directed path from $s_{i}$ to $t_{j}$ (if exists) which passing through $H_{ij}^{k}$ must enter $H_{ij}^{k}$ at $x_{1,ij}^{k}$ or $x_{2,ij}^{k}$, and leave $H_{ij}^{k}$ at $y_{1,ij}^{k}$ or $y_{2,ij}^{k}$. 
    
    Thus, if $I$ is a positive instance of Directed 2-Linkage, then we can find the following $N$ vertex-disjoint trees $T_{1},T_{2},\cdots,T_{N}$ such that $T_{i}$ is a pendant $(S_{i},s_{i})$-tree for each $i\in [N]$: 
    \begin{itemize}
        \item 
        $\begin{aligned}[t]
            V(T_{1})=&S_{1}\cup \{u_{1,j}\mid2\leq j\leq N\} \\
            &\cup \{x\mid x\in V(P_{1,1j}^{k}),j\in[N-1],2\leq k\leq N,j\neq k\},\\
            A(T_{1})=&\{s_{1}x_{1,11}^{2},u_{1,2}x_{1,12}^{3}\}\cup \{y_{1,1j}^{N}u_{1,j+1},u_{1,j+1}t_{j+1}\mid j\in [N-1]\}\\
            &\cup \{u_{1,j}x_{1,1j}^{2},y_{1,1j}^{j-1}x_{1,1j}^{j+1}\mid 3\leq j\leq N-1\}\\
            &\cup \{y_{1,1j}^{k}x_{1,1j}^{k+1}\mid j\in [N-1],2\leq k\leq N-1,j\neq k,k+1\}\\
            &\cup \{xy\mid xy\in A(P_{1,1j}^{k}),j\in[N-1],2\leq k\leq N,j\neq k\};
        \end{aligned}$
        \item For each $2\leq i\leq N-1$:\\
        $\begin{aligned}[t]
            V(T_{i})=&S_{i}\cup \{u_{i,j}\mid j\in [N],j\neq i\}\\
            &\cup \{x\mid x\in V(P_{1,ij}^{k}),j\in[N-1],i+1\leq k\leq N,j\neq k\}\\
            &\cup \{x\mid x\in V(P_{2,hj}^{i}),h\in[i-1],j\in[N-1],j\neq i\},\\
            A(T_{i})=&\{s_{i}u_{i,1},y_{1,i(i-1)}^{N}x_{1,ii}^{i+1},u_{i,i+1}x_{1,i(i+1)}^{i+2},u_{i,N}t_{N}\}\\
            &\cup\{u_{i,j}x_{1,ij}^{i+1}\mid j\in [N-1],j\neq i,i+1\}\\
            &\cup\{u_{i,j}x_{2,(i-1)j}^{i}\mid j\in [N-1],j\neq i\}\\
            &\cup\{y_{1,ij}^{N}u_{i,j+1}\mid j\in [N-1],j\neq i-1\}\\
            &\cup\{y_{1,ij}^{k}x_{1,ij}^{k+1}\mid j\in[N-1],i+1\leq k\leq N-1,j\neq k,k+1\}\\
            &\cup\{y_{1,ij}^{j-1}x_{1,ij}^{j+1}\mid i+2\leq j\leq N-1\}\\
            &\cup\{y_{2,(h+1)j}^{i}x_{2,hj}^{i},y_{2,1j}^{i}t_{j}\mid h\in [i-2],j\in [N-1],j\neq i\}\\
            &\cup \{xy\mid xy\in A(P_{1,ij}^{k}),j\in[N-1],i+1\leq k\leq N,j\neq k\}\\
            &\cup \{xy\mid xy\in A(P_{2,hj}^{i}),h\in [i-1],j\in[N-1],j\neq i\};
        \end{aligned}$
        \item 
        $\begin{aligned}[t]
            V(T_{N})=&S_{N}\cup \{u_{N,j}\mid j\in [N-1]\}\cup  \{x\mid x\in V(P_{2,hj}^{N}),h,j\in [N-1]\},\\
            A(T_{N})=&\{s_{N}u_{N,1}\}\cup\{u_{N,j}u_{N,j+1}\mid j\in [N-2]\}\\
            &\cup\{u_{N,j}x_{2,(N-1)j}^{N},y_{2,(h+1)j}^{N}x_{2,hj}^{N},y_{2,1j}^{N}t_{j}\mid h\in [N-2],j\in [N-1]\}\\
            &\cup \{xy\mid xy\in A(P_{2,hj}^{N}),h,j\in [N-1]\}.\\
        \end{aligned}$
        
    \end{itemize}


    We now construct a digraph $D$ from $D'$ by making the following modifications: 
    \begin{itemize}
        \item Add a new vertex $r$ and create $N$ new arcs from $r$ to $s_{i}$ for all $i\in [N]$.
        \item Create $N$ new arcs $s_{i}t_{i}$ for all $i\in[N]$.
        \item Add new arcs $s_{1}u_{2,1},u_{2,1}u_{3,1},\cdots,u_{N-1,1}u_{N,1}$. 
        \item For each odd integer $3\leq q\leq N-1$, add new arcs $u_{q,1}u_{q-1,2},u_{q-1,2}u_{q-2,3},\\ \cdots,u_{\frac{q+1}{2}+1,\frac{q+1}{2}-1}u_{\frac{q+1}{2}-1,\frac{q+1}{2}+1},\cdots,u_{2,q-1}u_{1,q}$. 
        \item For each even integer $q\in [N-1]$, add new arcs $u_{q,1}u_{q-1,2},u_{q-1,2}u_{q-2,3},\\ \cdots,u_{\lfloor\frac{q+1}{2}\rfloor+1,\lfloor\frac{q+1}{2}\rfloor}u_{\lfloor\frac{q+1}{2}\rfloor,\lfloor\frac{q+1}{2}\rfloor+1},\cdots,u_{2,q-1}u_{1,q}$.
    \end{itemize}
    
    Let $S=\{r\}\cup S_{\beta}$. To complete our argument, we still need to prove the following two lemmas. 

    \begin{claim}
        If $I$ is a positive instance of Directed 2-Linkage, then $D$ has $N$ internally-disjoint pendant $(S,r)$-trees.
    \end{claim}

    \begin{proof}
        If $I$ is a positive instance of Directed 2-Linkage, then there exist $N$ vertex-disjoint trees $T_{1},T_{2},\cdots,T_{N}$, where $T_{i}$ is the pendant $(S_{i},s_{i})$-tree established earlier. For each $i\in [N]$, we construct a pendant $(S,r)$-tree $\hat{T_{i}}$ with vertex set $V(\hat{T_{i}})=\{r,t_{i}\}\cup V(T_{i})$ and arc set $A(\hat{T_{i}})=\{rs_{i},s_{i}t_{i}\}\cup A(T_{i})$. One can verify that these $N$ pendant $(S,r)$-trees are pairwise internally-disjoint. This completes the proof of Claim 1.
    \end{proof}

    \begin{claim}
        If $I$ is a negative instance of Directed 2-Linkage, then $D$ has exactly one internally-disjoint pendant $(S,r)$-tree.
    \end{claim}

    \begin{proof}
        By the construction of the digraph $D$, we can always find one pendant $(S,r)$-tree $T^{*}$ formed by the arcs in
        \begin{align*}
            &\{rs_{1},s_{1}u_{2,1},s_{1}t_{1}\}\cup \{u_{i+1,1}u_{i+2,1}\mid i\in [N-2]\}\\
            \cup &\{u_{k,q+1-k}u_{k-1,q-k+2}\mid 2\leq k\leq q\leq N-1,k\neq \frac{q+1}{2},\frac{q+1}{2}+1\}\\
            \cup &\{u_{k,q-k+1}u_{k-2,q-k+3}\mid 3\leq q\leq N-1,k= \frac{q+1}{2}+1\},
        \end{align*}
        where $k,q$ are both integers.


        Thus, it remains to prove that there are not two internally-disjoint pendant $(S,r)$-trees. Otherwise, we arbitrarily choose two such trees, denoted as $\tilde{T_{1}}$ and $\tilde{T_{2}}$. Let $rs_{a}\in A(\tilde{T_{1}})$ and $rs_{b}\in A(\tilde{T_{2}})$, where $1\leq a\neq b\leq N$. 
        We just consider the case that $a< b$ here (as the argument for the remaining case is similar), and we aim to show that there does not exist a directed path from $s_{a}$ to $t_{N}$ in $\tilde{T_{1}}$ and a directed path from $s_{b}$ to $t_{b-1}$ in $\tilde{T_{2}}$ simultaneously.

        We proceed by induction on $b$. For the base case, let $b=2$ (so now $a=1$), and suppose there exists a directed path $\tilde{P}$ from $s_{1}$ to $t_{N}$ in $\tilde{T_{1}}$. Let $\tilde{Q}$ denote the directed path from $s_{2}$ to $t_{1}$ in $\tilde{T_{2}}$. From the earlier construction of digraph $D$, we know that $s_{2}u_{2,1}\in A(\tilde{Q})$. Consequently, the directed path $\tilde{P}$ in $\tilde{T_{1}}$ must enter each $H_{11}^{k}$ at $x_{1,11}^{k}$ and leave it at $y_{1,11}^{k}$ for all $2\leq k\leq N$. On the other hand, the directed path $\tilde{Q}$ in $\tilde{T_{2}}$ must enter some $H_{11}^{\ell}$ at $x_{2,11}^{\ell}$ and leave it at $y_{2,11}^{\ell}$ for some $2\leq \ell\leq N$ (otherwise, $\tilde{P}$ and $\tilde{Q}$ would not be vertex-disjoint). However, this implies that there exist two vertex-disjoint directed paths $P_{1,11}^{\ell}, P_{2,11}^{\ell}$ in $H_{11}^{\ell}$, where $P_{i,11}^{\ell}$ starts at $x_{i,11}^{\ell}$ and ends at $y_{i,11}^{\ell}$ for all $i\in[2]$, which produces a contradiction, since $I$ is a negative instance of Directed 2-Linkage. 

        For the induction step, let $b>2$ (so now $1\leq a\leq b-1$), and suppose that there exists a path $\tilde{P}$ from $s_{a}$ to $t_{N}$ in $\tilde{T_{1}}$. Let $\tilde{Q}$ denote the directed path from $s_{b}$ to $t_{b-1}$ in $\tilde{T_{2}}$. It is not hard to see that $\tilde{Q}$ cannot exist without using arcs from a copy of $H$. Note that $s_{b}u_{b,1}\in A(\tilde{Q})$. To ensure that $\tilde{P}$ and $\tilde{Q}$ are vertex-disjoint, the directed path $\tilde{P}$ must contain the directed subpath from $u_{a,1}$ to $u_{1,a}$ without using any arcs from a copy of $H$. If not, we can consider the directed path from $s_{b}$ to $t_{b-2}$ and $\tilde{P}$, and reach the same contradiction. Consequently, $\tilde{P}$ must enter each $H_{1,b-1}^{k}$ at $x_{1,1(b-1)}^{k}$ and leave it at $y_{1,1(b-1)}^{k}$ for all $b\leq k\leq N$. On the other hand, the directed path $\tilde{Q}$ in $\tilde{T_{2}}$ must enter some $H_{1,b-1}^{\ell}$ at $x_{2,1(b-1)}^{\ell}$ and leave it at $y_{2,1(b-1)}^{\ell}$ for some $b\leq \ell\leq N$. However, this implies that there exist two vertex-disjoint directed paths $P_{1,1(b-1)}^{\ell}, P_{2,1(b-1)}^{\ell}$ in $H_{1,b-1}^{\ell}$, where $P_{i,1(b-1)}^{\ell}$ starts at $x_{i,1(b-1)}^{\ell}$ and ends at $y_{i,1(b-1)}^{\ell}$ for all $i\in[2]$, which produces a contradiction, since $I$ is a negative instance of Directed 2-Linkage.
        As an example, let $a=N-1$ and $b=N$. To avoid repeated vertices, the directed path $\tilde{P}$ is the union of the arc $s_{N}u_{N-1,1}$, the directed path $L$ consisting of arcs 
        \begin{align*}
            &\{u_{k,q+1-k}u_{k-1.q-k+2}\mid 2\leq k\leq N-1,k\neq \frac{N}{2},\frac{N}{2}+1\}\\
            \cup &\{u_{k,q+1-k}u_{k-2.q-k+3}\mid k=\frac{N}{2}+1\},
        \end{align*}
        and the directed path from $u_{1,N-1}$ to $t_{N}$, the latter of which must pass through each $H_{1(N-1)}^{k}$ for $2\leq k\leq N$, as illustrated by the green dotted-line rectangle in Figure \ref{fig:digraph D'}. Moreover, the directed path $\tilde{Q}$ is required to traverse all $H_{i(N-1)}^{N}$ for $i\in [N-1]$, depicted by the red dotted-line rectangle in Figure \ref{fig:digraph D'}. Hence, we find that there exist two paths crossing each other in $H_{1(N-1)}^{N}$ without sharing any common vertices. It should be noted that even if we consider other ways of constructing the directed path $\tilde{P}$, we will always find an $H_{ij}^{k}$ that contains two vertex-disjoint paths, i.e. $P_{1,ij}^{k}$ and $P_{2,ij}^{k}$. In all cases, this contradicts the fact that $I$ is a negative instance of Directed 2-Linkage. Thus, Claim 2 is proved.
    \end{proof}

    Based on the above two claims, the positive instance of Directed 2-Linkage provides $N$ internally-disjoint pendant $(S,r)$-trees, the negative instance of Directed 2-Linkage provides exactly one internally-disjoint pendant $(S,r)$-tree. It implies that the NP-hard gap $[1+\epsilon,2]$ of Directed 2-Linkage is amplified to a bigger NP-hard gap $[1+\epsilon,N]$. Note that the number of copies of $H$ is $O(N^{3})$ where $N=|V(H)|^{1/\epsilon}$, we obtained the number of vertices in $D$ is $O(N^{3+\epsilon})$. Hence, we obtained a gap with $\Omega(n^{{1/3}-\epsilon})$. This completes the proof of the theorem.     
\end{proof}

\section{Bounds for $\tau_{k}(D)$}

In this section, we first give some basic results on the parameter $\tau_{k}(D)$ which will be used in the following argument.

\begin{proposition}\label{pro3.1}
    Let $D$ be a digraph with $n$ vertices, and let $k$ be an integer such that $2\leq k\leq n$. The following assertions hold:

    (i) $\tau_{k}(H)\leq \tau_{k}(D)$, where $H$ is a spanning subdigraph of $D$.

    (ii) $\tau_{k}(D)\leq \min\{\delta^{+}(D),\delta^{-}(D)\}$.

    (iii) $\tau_{k}(D)=0$ for $k\geq \max\{\delta^{0}(D)+1,3\}$.
\end{proposition}

\begin{proof}
    Part (i) and (ii) can be easily verified by the definition of $\tau_{k}(D)$.
    
    We now prove part(iii). The argument is divided into the following two cases:

    Case 1: $\delta^{+}(D)\leq \delta^{-}(D)$. We first consider the case that $\delta^{+}(D)=1$.  Assume that $\delta_{D}^{+}(r)=1$ and let $N_{D}^{+}(r)=\{x\}$. Choose $S \subseteq V(D)$ such that $|S|=k\geq 3$ and $\{r,x\}\subseteq S$. Clearly, $\tau_{S,r}(D)=0$, and so $\tau_{k}(D)=0$ for $k\geq 3$. 
    
    We next consider the case that $\delta^{+}(D)=t\geq 2$. Assume that $\delta_{D}^{+}(r)=t$. We choose $S\subseteq V(D)$ such that $|S|=k\geq t+1$ and $(\{r\} \cup N_{D}^{+}(r))\subseteq S$. Clearly, $\tau_{S,r}(D)=0$, and so $\tau_{k}(D)=0$ for $k\geq t+1$.

    Case 2: $\delta^{+}(D)> \delta^{-}(D)$. We first consider the case that $\delta^{-}(D)=1$. Assume that $\delta_{D}^{-}(u)=1$ and let $N_{D}^{-}(u)=\{r\}$. Choose $S\subseteq V(D)$ such that $|S|=k\geq 3$ and $\{r,u\}\subseteq S$. Clearly, $\tau_{S,r}(D)=0$, and so $\tau_{k}(D)=0$ for $k\geq 3$. 
    
    We next consider the case that $\delta^{-}(D)=t\geq 2$. Assume that $\delta_{D}^{-}(u)=t$ and let $r\in N_{D}^{-}(u)$. We choose $S\subseteq V(D)$ such that $|S|=k\geq t+1$ and $(\{u\}\cup N_{D}^{-}(u))\subseteq S$. Clearly, $\tau_{S,r}(D)=0$, and so $\tau_{k}(D)=0$ for $k\geq t+1$.

    According to the above argument, the conclusion holds.
\end{proof}

Now we give a sharp lower and a sharp upper bound for $\tau_k(D)$ on a general digraph $D$.

\begin{theorem}\label{thm3.1}
    Let $k$, $n$ be two integers with $3\leq k\leq n$, and let $D$ be a digraph of order $n$. We have $$0\leq \tau_{k}(D)\leq n-k.$$ Moreover, both bounds are sharp. In
particular, the upper bound holds if and only if $D\cong \overleftrightarrow{K}_{n}$. 
\end{theorem}

\begin{proof}
The lower bound is clear. For its sharpness, let $T$ be an orientation of a star $K_{1,n-1}$ of order $n$ such that the out-degree of one vertex is $n-1$ and the out-degrees of  all remaining vertices are 0. By Proposition~\ref{pro3.1}(ii) and (iii), $\tau_{k}(T)=0$ for each $3\leq k\leq n$.

We now prove the upper bound and characterize the extremal graphs.
Let $H\cong \overleftrightarrow{K}_{n}$ such that $V(H)=\{u_{i}, v_{j}\mid i\in [k], j\in [n-k]\}$. Without loss of generality, assume that $S=\{u_{i}\mid i\in [k]\}$, choose $u_{k}$ as the root. We construct $n-k$ pairwise internally-disjoint pendant $(S,u_{k})$-trees: $T_{i}~(i\in [n-k])$ such that $A(T_{i})=\{u_{k}v_{i},v_{i}u_{1},v_{i}u_{2},\cdots,v_{i}u_{k-1}\}$, which means that $\tau_{S, u_k}(H)\geq n-k$. Furthermore, we have $\tau_{k}(H)\geq n-k$. On the other hand, by the definition of a pendant tree, there are at most $n-k$ out-arcs of $u_k$ which can be used in pendant $(S,u_{k})$-trees, that is, $\tau_{S, u_k}(H)\leq n-k$, so we have $\tau_{k}(H)\leq n-k$. Thus, $\tau_{k}(\overleftrightarrow{K}_{n})=n-k$. 

Now let $H'\cong H-e$, where $e\in A(H)$. Without loss of generality, assume that $e=u_{k}v_{n-k}$. By the definition of a pendant tree, there are at most $n-k-1$ out-arcs of $u_k$ which can be used in pendant $(S,u_{k})$-trees of $H'$, that is, $\tau_{S, u_k}(H')\leq n-k-1$, so we have $\tau_{k}(\overleftrightarrow{K}_{n}-e)\leq n-k-1$.

Now we have $\tau_{k}(\overleftrightarrow{K}_{n})=n-k$ and $\tau_{k}(\overleftrightarrow{K}_{n}-e)\leq n-k-1$ for any arc $e$ of $\overleftrightarrow{K}_{n}$. By Proposition~~\ref{pro3.1}(i), for a general digraph $D$, we have $\tau_{k}(D)\leq  \tau_{k}(\overleftrightarrow{K}_{n})=n-k.$ Moreover, this bound holds if and only if $D\cong \overleftrightarrow{K}_{n}$.
\end{proof}

Next, we give a sharp upper bound for $\tau_{k}(D)$ in terms of arc-cuts of $D$. For two disjoint vertex subsets $A,B\subseteq V(D)$, we use $(A, B)_{D}$ to denote the set of arcs from $A$ to $B$ in $D$.

\begin{theorem}\label{thm3.2}
    Let $k$, $n$ be two integers with $3\leq k\leq n$. Let $D$ be a digraph of order $n$. We have
    \begin{align*}
        \tau_{k}(D)\leq \min\limits_{S\subseteq V(D),|S| = k,r\in S} \lfloor \frac{1}{k-1}|(\overline{S},S\backslash\{r\})_{D}| \rfloor,
    \end{align*}
    where $\overline{S}=V(D)\backslash S$. Moreover, the bound is sharp.
\end{theorem}
\begin{proof}
Let $S\subseteq V(D)$ with $|S|=k$ and $r\in S$.  Let $\mathcal{T}=\{T_{i}\mid i\in [\ell]\}$ be a set of $\ell$ pairwise internally-disjoint pendant $(S,r)$-trees, where $\ell =\tau_{S,r}(D)$. Observe that for each $T_{i}\in \mathcal{T}$, we have
    \begin{align*}
        |A(T_{i})\cap (\overline{S},S\backslash\{r\})_{D}|\geq k-1.
    \end{align*}
    Therefore,
    \begin{align*}
        (k-1){\ell}\leq \sum\limits_{i=1}^{\ell}|A(T_{i})\cap (\overline{S},S\backslash\{r\})_{D}|\leq |(\overline{S},S\backslash\{r\})_{D}|.
    \end{align*}

    Hence,
    \begin{align*}
        \tau_{S,r}(D)=\ell\leq \lfloor \frac{1}{k-1}|(\overline{S},S\backslash\{r\})_{D}| \rfloor.
    \end{align*}
    By the arbitrariness of $S$ and $r$, we have
    \begin{align*}
        \tau_{k}(D)&=\min\limits_{S\subseteq V(D),|S|=k,r\in S}\tau_{S,r}(D)\\
        &\leq \min\limits_{S\subseteq V(D),|S| = k,r\in S} \lfloor \frac{1}{k-1}|(\overline{S},S\backslash\{r\})_{D}| \rfloor.
    \end{align*}
    To show the sharpness of the bound, we consider the complete digraph $D=\overleftrightarrow{K_{n}}$. For any $S\subseteq V(D)$ with $|S|=k$ and $r\in S$, we have $$|(\overline{S},S\backslash\{r\})_{D}|=(k-1)(n-k),$$ and so $$\frac{1}{k-1}|(\overline{S},S\backslash\{r\})_{D}|=n-k.$$ Furthermore, the following holds:
    $$\min\limits_{S\subseteq V(D),|S| = k,r\in S} \lfloor \frac{1}{k-1}|(\overline{S},S\backslash\{r\})_{D}| \rfloor=n-k.$$ On the other hand, $\tau_{k}(D)=n-k$ by Theorem~\ref{thm3.2}. Therefore, when $D=\overleftrightarrow{K_{n}}$, 
$$\tau_{k}(D)=n-k=\min\limits_{S\subseteq V(D),|S| = k,r\in S} \lfloor \frac{1}{k-1}|(\overline{S},S\backslash\{r\})_{D}| \rfloor.$$
\end{proof}

Given a graph parameter $f(G)$, the Nordhaus-Gaddum problem is to determine sharp bounds for (1) $f(G)+f(G^{c})$ and (2) $f(G)f(G^{c})$, where $G^c$ is the complement of $G$. 
The Nordhaus-Gaddum type relations have received wide attention, and the readers can see \cite{Aouchiche-Hansen} for a survey. The following Theorem gives the sharp Nordhaus-Gaddum-type bounds for the parameter $\tau_{k}(D)$.

\begin{theorem}\label{thm3.3}
    Let $k$ and $n$ be two integers with $3\leq k\leq n$. Let $D$ be a general digraph of order $n$. The following assertions hold:

    (i) $0\leq \tau_{k}(D)+\tau_{k}(D^{c})\leq n-k$. Moreover, both bounds are sharp.

    (ii) $0\leq \tau_{k}(D)\tau_{k}(D^{c})\leq \lfloor (\frac{n-k}{2})^{2} \rfloor$. Moreover, the lower bound is sharp.
\end{theorem}

\begin{proof}
    Part (i). The lower bound holds by Theorem \ref{thm3.1}.
    Since $D\cup D^{c}=\overleftrightarrow{K}_{n}$ and $V(D)=V(D^{c})=V(\overleftrightarrow{K}_{n})$, for any $k$-subset $S\subseteq V(D)$ and $r\in S$, we have $\tau_{S,r}(D)+\tau_{S,r}(D^{c})\leq \tau_{S,r}(\overleftrightarrow{K}_{n})$. By Theorem \ref{thm3.1}, we have
    \begin{align*}
        n-k=\tau_{k}(\overleftrightarrow{K}_{n}) 
        &=\min\limits_{S\in V(\overleftrightarrow{K}_{n}),|S|=k,r\in S}\tau_{S,r}(\overleftrightarrow{K}_{n})\\
        &\geq \min\limits_{S\in V(D),|S|=k,r\in S}(\tau_{S,r}(D)+\tau_{S,r}(D^{c}))\\
        &\geq \tau_{k}(D)+\tau_{k}(D^{c}).
    \end{align*}
    Hence, the upper bound also holds.

    Part (ii). The lower bound holds by Theorem \ref{thm3.1}. For the upper bound, we have
    \begin{align*}
        \tau_{k}(D)\tau_{k}(D^{c})\leq(\frac{\tau_{k}(D)+\tau_{k}(D^{c})}{2})^{2}\leq (\frac{n-k}{2})^{2}.
    \end{align*}

    To show the sharpness, let $D=\overleftrightarrow{K}_{n}$, then we have $\tau_{k}(D)=n-k$ and $\tau_{k}(D^{c})=0$. We obtain $\tau_{k}(D)+\tau_{k}(D^{c})=n-k$ and $\tau_{k}(D)\tau_{k}(D^{c})=0$ which implies the upper bound in (i) and lower bound in (ii) are both sharp.

    Let $D$ be a nonstrong tournament, in which case it is not difficult to see that $D^{c}$ is also nonstrong. Then we have $\tau_{k}(D)=\tau_{k}(D^{c})=0$, which can serve as an example to illustrate the lower bound in (i) is sharp.
\end{proof}

\vskip 1cm

\noindent {\bf Acknowledgement.} This work was supported by National Natural Science Foundation of China under Grant No. 12371352 and Yongjiang Talent Introduction Programme of Ningbo under Grant No. 2021B-011-G.

\end{document}